\documentclass[11pt]{article}
\usepackage{palatino}
\usepackage{amsmath}
\usepackage{amsthm}
\usepackage{amssymb}
\usepackage{latexsym}
\usepackage{pstricks}
\usepackage{graphicx, pst-plot, pst-node, pst-text, pst-tree}
\usepackage{titlefoot}
\usepackage{titlesec}
\usepackage[small,it]{caption}
\usepackage{pigpen}

\usepackage{color}
\definecolor{lightgray}{rgb}{0.8, 0.8, 0.8}
\definecolor{darkgray}{rgb}{0.7, 0.7, 0.7}
\definecolor{darkblue}{rgb}{0, 0, .4}

\usepackage[bookmarks]{hyperref}
\hypersetup{
        colorlinks=true,
        linkcolor=darkblue,
        anchorcolor=darkblue,
        citecolor=darkblue,
        urlcolor=darkblue,
        pdfpagemode=UseThumbs,
        pdftitle={Counting (3+1)-avoiding permutations},
        pdfsubject={Combinatorics},
        pdfauthor={Atkinson, Sagan, and Vatter},
}

\newtheorem{theorem}{Theorem}[section]
\newtheorem{proposition}[theorem]{Proposition}
\newtheorem{lemma}[theorem]{Lemma}

\newtheorem{corollary}[theorem]{Corollary}

\newcounter{todocounter}

\newcommand{\Av}{\operatorname{Av}}

\newcommand{\C}{\mathcal{C}}
\newcommand{\D}{\mathcal{D}}

\newcommand{\tilenla}{\raisebox{-2pt}{\pigpenfont A}}
\newcommand{\tilenlb}{\raisebox{-2pt}{\pigpenfont J}}
\newcommand{\tilelna}{\raisebox{-2pt}{\pigpenfont I}}
\newcommand{\tilelnb}{\raisebox{-2pt}{\pigpenfont R}}

\newcommand{\1}{{\bf 1}}
\newcommand{\2}{{\bf 2}}
\newcommand{\3}{{\bf 3}}

\newcommand{\ba}{{\bf a}}
\newcommand{\bb}{{\bf b}}

\makeatletter
\newfont{\footsc}{cmcsc10 at 8truept}
\newfont{\footbf}{cmbx10 at 8truept}
\newfont{\footrm}{cmr10 at 10truept}
\renewenvironment{abstract}%
                {
                  \begin{list}{}%
                     {\setlength{\rightmargin}{1in}%
                      \setlength{\leftmargin}{1in}}%
                   \item[]\ignorespaces\begin{small}}%
                 {\end{small}\unskip\end{list}}

\amssubj{05A05, 05A15}
\keywords{algebraic generating function, permutation class, restricted permutation, simple permutation, substitution decomposition, $(3+1)$-free}

\newpagestyle{main}[\small]{
        \headrule
        \sethead[\usepage][][]
        {\sc Counting $\mathbf{(3+1)}$-Avoiding Permutations}{}{\usepage}}

\title{\sc Counting $\mathbf{(3+1)}$-Avoiding Permutations}
\author{%
M. D. Atkinson\\[-0.25ex]
\small Department of Computer Science\\[-0.5ex]
\small University of Otago\\[-0.5ex]
\small Dunedin, New Zealand\\[-0.5ex]
\small {\tt mike@cs.otago.ac.nz}\\[1.5ex]
Bruce E. Sagan\\[-0.25ex]
\small Department of Mathematics\\[-0.5ex]
\small Michigan State University\\[-0.5ex]
\small East Lansing, Michigan, USA\\[-0.5ex]
\small {\tt sagan@math.msu.edu}\\[1.5ex]
Vincent Vatter\\[-0.25ex]
\small Department of Mathematics\\[-0.5ex]
\small University of Florida\\[-0.5ex]
\small Gainesville, Florida, USA\\[-0.5ex]
\small {\tt vatter@ufl.edu}\\[1.5ex]
}

\titleformat{\section}
        {\large\sc}
        {\thesection.}{1em}{}   

\date{}

\begin{document}
\maketitle

\pagestyle{main}

\begin{abstract}
A poset is {\it $(\3+\1)$-free\/} if it contains no induced subposet isomorphic to the disjoint union of a 3-element chain and a 1-element chain.  These posets are of interest because of their connection with interval orders and their appearance in the $(\3+\1)$-free Conjecture of Stanley and Stembridge.   The dimension 2 posets $P$ are exactly the ones which have an associated permutation $\pi$ where $i\prec j$ in $P$ if and only if $i<j$ as integers and $i$ comes before $j$ in the one-line notation of $\pi$.  So we say that a permutation $\pi$ is {\it $(\3+\1)$-free\/} or {\it $(\3+\1)$-avoiding\/} if its poset is $(\3+\1)$-free.  This is equivalent to $\pi$ avoiding the permutations $2341$ and $4123$ in the language of pattern avoidance.  We give a complete structural characterization of such permutations.  This permits us to find their generating function.
\end{abstract}

\section{Introduction}

The permutation $\pi$ of $[n]=\{1,2,\dots,n\}$ {\it contains\/} the permutation $\sigma$ of $[k]$
if $\pi$ has a subsequence of length $k$ order isomorphic to $\sigma$, and such a subsequence is called an {\it occurrence\/}, or {\it copy\/}, of $\sigma$.  For example, $\pi=391867452$ (written in list, or one-line notation) contains $\sigma=51342$, as can be seen by considering the subsequence $91672$ ($=\pi(2),\pi(3),\pi(5),\pi(6),\pi(9)$).  If $\pi$ does not contain $\sigma$ we say that $\pi$ \emph{avoids} $\sigma$. A {\it permutation class\/}, sometimes abbreviated to simply \emph{class}, is a downset of permutations under this order; thus if $\C$ is a permutation class, $\pi\in\C$, and $\pi$ contains $\sigma$, then $\sigma\in\C$.  Every permutation class can be described by the minimal permutations which are \emph{not} in the class.  We call such a set a \emph{basis}, and denote by $\Av(B)$ the class with basis $B$.

Given a class $\C$, we denote by $\C_n$ the set of permutations in $\C$ of length $n$.  It is natural to ask for the enumeration of $\C$ and this is usually answered in terms of its generating function,
$$
\sum_{\pi\in\C} x^{|\pi|}=\sum_{n\ge 0} |\C_n|x^n.
$$
The class we will consider in this paper is motivated by ideas in the theory of posets (partially ordered sets).  Call a poset {\it $(\ba+\bb)$-free\/} if it contains no induced subposet isomorphic to a disjoint union of an $a$-element chain and a $b$-element chain.  Fishburn~\cite{intransitive:indifference:} characterized the $(\2+\2)$-free posets as those which could be modeled by intervals of real numbers, where we let $[a,b]<[c,d]$ if and only if $b<c$ so that $[a,b]$ is completely to the left of $[c,d]$ on the real line.  He also characterized the posets which are both $(\2+\2)$- and $(\3+\1)$-free as those interval orders where all the intervals have length one.  But until more recently there has been no characterization of $(\3+\1)$-free posets.  These posets also come up in the $(\3+\1)$-free Conjecture of Stembridge and Stanley~\cite{stanley:on-immanants-of:}.  Stanley~\cite{stanley:a-symmetric:function:generalization:} defined a symmetric function generalization $X_G$ of the chromatic polynomial of a graph $G$.  The conjecture in question states that if one takes the incomparability graph $G$ of a $(\3+\1)$-free poset (making two vertices adjacent in the graph if the corresponding elements are incomparable in the poset) and expresses $X_G$ in the elementary symmetric function basis, then all the coefficients are nonnegative.  To date there has been only partial progress on this question by Gasharov~\cite{incomparability:graphs:}, Gebhard and Sagan~\cite{a:chromatic:symmetric:function:}, and Lee and Sagan~\cite{an:algorithmic:sign:reversing:}.

Every permutation $\pi$ gives rise to a poset $P_\pi$ by letting $i\prec j$ in $P_\pi$ if and only if $i<j$ and $i$ appears to the left of $j$ in $\pi$.  The posets arising this way are exactly those of dimension 2.  Call a permutation $\pi$ {\it $(\3+\1)$-free\/} or {\it $(\3+\1)$-avoiding\/} if its poset is $(\3+\1)$-free.   Note that $P_\pi=\3+\1$ precisely when $\pi=2341$ or $4123$.  So the class of $(\3+\1)$-free permutations is $\Av(2341,4123)$.  In this paper we completely characterize the elements in this class.  Using this characterization, we are able to compute the corresponding generating function.  The hope is that this viewpoint might also be useful in making progress on the $(\3+\1)$-free Conjecture in the case of dimension 2 posets.  We should  mention that Skandera~\cite{a:characterization:of:three:plus:one:free:} has a useful characterization of all $(\3+\1)$-free posets.  But his involves conditions on the entries of the square of the antiadjacency matrix of the poset and so seems to be quite different from ours.


A secondary motivation for enumerating the class $\Av(2341,4123)$ is that it belongs to a family of classes which have proved to be a fertile testing ground for different enumerative techniques.  For bases $B$ consisting of a single permutation of length at most $4$, exact enumerations for $\Av(B)$ are known except in the notable case of $B=\{1324\}$ (or its symmetry, $B=\{4231\}$).  (Here only lower and upper bounds are known, see Albert, Elder, Rechnitzer, Westcott, and Zabrocki~\cite{albert:on-the-wilf-sta:} and B\'ona~\cite{bona:the-exponential-up:}.)  For bases $B$ consisting of two permutations, exact enumerations are known in the case where one element of $B$ has length at most $3$ and the other has length at most $4$.  However, in the case where $B$ consists of two permutations of length $4$, much less is known.

The permutation containment relation is invariant under the $8$ symmetries generated by reversal, complementation, and inversion.  These symmetries can be used to cut down the number of cases; in particular, the ${4!\choose 2}$ different sets $B$ consisting of two permutations of length $4$ split into $56$ different \emph{symmetry classes}.  Of these $56$ essentially different classes, it is known that there are $38$ different enumerations, which follows from a long string of papers \cite{bona:the-permutation:,kremer:permutations-wi:, kremer:postscript:-per:, kremer:finite-transiti:, le:wilf-classes-of:}.  Only about half of these have been enumerated.

The approach we use here to enumerate $\Av(2341,4123)$ is based on simple permutations, so we briefly recall the salient definitions and properties.  An \emph{interval} of a permutation $\pi=\pi(1)\pi(2)\cdots \pi(n)$ is a contiguous subsequence $\pi(i)\pi(i+1)\cdots \pi(j)$ whose values form a contiguous set of integers. If a permutation has no intervals except for itself and its singletons then it is said to be \emph{simple}.  For example, $871329456$ has nontrivial intervals $87$, $132$, and $456$, while $31524$ is simple.  Figure~\ref{fig-example-simples} shows the plots of three further simple permutations; in this diagram and, in subsequent similar diagrams, the dots are placed at cartesian coordinates $(i,\pi(i))$.

\begin{figure}
\begin{center}
\begin{tabular}{ccccc}
\psset{xunit=0.01in, yunit=0.01in}
\psset{linewidth=0.005in}
\begin{pspicture}(0,0)(100,100)
\psline[linecolor=darkgray,linestyle=solid,linewidth=0.02in](55,0)(55,105)
\psaxes[dy=10,Dy=1,dx=10,Dx=1,tickstyle=bottom,showorigin=false,labels=none](0,0)(100,100)
\pscircle*(10,90){0.04in}
\pscircle*(20,70){0.04in}
\pscircle*(30,50){0.04in}
\pscircle*(40,30){0.04in}
\pscircle*(50,10){0.04in}
\pscircle*(60,100){0.04in}
\pscircle*(70,80){0.04in}
\pscircle*(80,60){0.04in}
\pscircle*(90,40){0.04in}
\pscircle*(100,20){0.04in}
\end{pspicture}
&\rule{10pt}{0pt}&
\psset{xunit=0.01in, yunit=0.01in}
\psset{linewidth=0.005in}
\begin{pspicture}(0,0)(100,100)
\psline[linecolor=darkgray,linestyle=solid,linewidth=0.02in](0,55)(105,55)
\psaxes[dy=10,Dy=1,dx=10,Dx=1,tickstyle=bottom,showorigin=false,labels=none](0,0)(100,100)
\pscircle*(10,50){0.04in}
\pscircle*(30,40){0.04in}
\pscircle*(50,30){0.04in}
\pscircle*(70,20){0.04in}
\pscircle*(90,10){0.04in}
\pscircle*(100,60){0.04in}
\pscircle*(80,70){0.04in}
\pscircle*(60,80){0.04in}
\pscircle*(40,90){0.04in}
\pscircle*(20,100){0.04in}
\end{pspicture}
&\rule{10pt}{0pt}&
\psset{xunit=0.01in, yunit=0.01in}
\psset{linewidth=0.005in}
\begin{pspicture}(0,0)(100,100)
\psaxes[dy=10,Dy=1,dx=10,Dx=1,tickstyle=bottom,showorigin=false,labels=none](0,0)(100,100)
\pscircle*(10,40){0.04in}
\pscircle*(20,10){0.04in}
\pscircle*(30,30){0.04in}
\pscircle*(40,70){0.04in}
\pscircle*(50,50){0.04in}
\pscircle*(60,20){0.04in}
\pscircle*(70,100){0.04in}
\pscircle*(80,80){0.04in}
\pscircle*(90,60){0.04in}
\pscircle*(100,90){0.04in}
\end{pspicture}
\end{tabular}
\end{center}
\caption[]{The two simple permutations on the left are the $123$-avoiding \emph{parallel alternations}.  The permutation shown on the right is another simple permutation in $\Av(2341,4123)$.}\label{fig-example-simples}
\label{parallelalternations}
\end{figure}
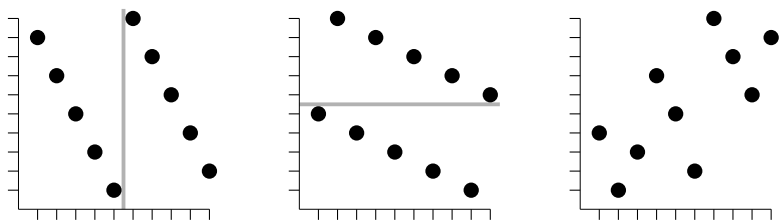

Simple permutations are precisely those that do not arise from a non-trivial inflation, in the following sense.  Let $\sigma$ be any permutation of length $m$ and $\alpha_1,\alpha_2,\ldots,\alpha_m$ any sequence of permutations.  Then the \emph{inflation} of $\sigma$ by $\alpha_1,\alpha_2,\ldots,\alpha_m$, denoted  $\sigma[\alpha_1,\alpha_2,\ldots,\alpha_m]$, is the permutation of length $|\alpha_1|+\cdots+|\alpha_m|$ which decomposes into $m$ segments $\alpha'_1\alpha'_2\cdots\alpha'_m$ where each segment $\alpha'_i$ is an interval which is order isomorphic to $\alpha_i$, and the sequence $a_1a_2\cdots a_n$ formed by any (and hence every) choice of $a_i$ from $\alpha'_i$ is order isomorphic to $\sigma$.  For example the inflation of $3142$ by $21, 132, 1, 123$ is
\[3142[21, 132, 1, 123]=87\ 132\ 9\ 456.\]

The precise connection between simple permutations and inflations is furnished by a result from Albert and Atkinson~\cite{albert:simple-permutat:}.

\begin{lemma}\label{lem-simple-inflate}
For every permutation $\pi$ there is a unique simple permutation $\sigma$ such that
$$
\pi=\sigma[\alpha_1,\alpha_2,\ldots,\alpha_m].
$$
Furthermore, except when $\sigma=12$ or $\sigma=21$, the intervals of $\sigma$ that correspond to $\alpha_1,\alpha_2,\ldots,\alpha_m$ are uniquely determined.  In the case that $\sigma=12$ (respectively $\sigma=21$), the intervals are unique so long as we require the first of the two intervals to be \emph{sum} (respectively, \emph{skew}) \emph{indecomposable}, which means that it cannot be decomposed further as a nontrivial inflation of $12$ (respectively, of $21$).
\end{lemma}

One important feature of $\Av(2341,4123)$ is that it is a \emph{sum closed} class, meaning that if $\sigma$ and $\pi$ lie in $\Av(2341,4123)$ then $12[\sigma,\pi]$ also lies in $\Av(2341,4123)$.  The generating function for any sum closed class  is easily seen to be $1/(1-g)$, where $g$ is the generating function for the non-empty sum indecomposable permutations in the class.

A final lemma which we use was proved by Atkinson~\cite{atkinson:restricted-perm:}.

\begin{proposition}\label{propAv1233412types}
Every permutation in $\Av(123, 3412)$ is either a horizontal or vertical juxtaposition of two decreasing permutations (see Figure~\ref{Av1233412types}).
\end{proposition}

\begin{figure}
\begin{footnotesize}
\begin{center}
\begin{tabular}{ccc}
\psset{xunit=0.007in, yunit=0.007in}
\psset{linewidth=0.005in}
\begin{pspicture}(0,-30)(120,60)
\multirput(0,0)(60,0){3}{\psline[linecolor=darkgray,linestyle=solid,linewidth=0.02in]{c-c}(0,0)(0,60)}
\multirput(0,0)(0,60){2}{\psline[linecolor=darkgray,linestyle=solid,linewidth=0.02in]{c-c}(0,0)(120,0)}
\rput[c](30,30){$\D$}
\rput[c](90,30){$\D$}
\end{pspicture}
&\rule{0.3in}{0pt}&
\psset{xunit=0.007in, yunit=0.007in}
\psset{linewidth=0.005in}
\begin{pspicture}(0,0)(60,120)
\multirput(0,0)(60,0){2}{\psline[linecolor=darkgray,linestyle=solid,linewidth=0.02in]{c-c}(0,0)(0,120)}
\multirput(0,0)(0,60){3}{\psline[linecolor=darkgray,linestyle=solid,linewidth=0.02in]{c-c}(0,0)(60,0)}
\rput[c](30,30){$\D$}
\rput[c](30,90){$\D$}
\end{pspicture}
\end{tabular}
\end{center}
\end{footnotesize}
\caption{The two types of permutations in $\Av(123, 3412)$.  Throughout this paper, $\D$ denotes the class of decreasing permutations, i.e., $\Av(21)$.}
\label{Av1233412types}
\end{figure}
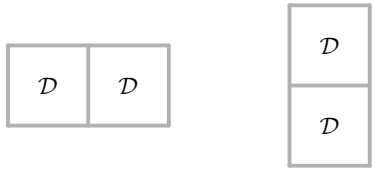

\section{Categories of Simple Permutations in $\Av(2341, 4123)$}

In this section we categorize simple permutations in $\Av(2341, 4123)$ according to whether they contain or avoid the permutations 123 and 3412.

\begin{proposition}\label{propsimpletypes}
Let $\sigma$ be any simple permutation of $\Av(2341, 4123)$ that contains both $123$ and $3412$.  Then $\sigma=5274163$.
\end{proposition}
\begin{proof}
Consider a simple permutation $\sigma\in\Av(2341,4123)$ that contains both 123 and 3412.  

Figure~\ref{propAv1233412types-structure-1} shows $\sigma$ with the $5\times 5$ grid defined by a copy of $3412$.  The unlabeled cells must be empty and the cells labeled $\D$ (with or without a subscript) must be decreasing; this is a direct consequence of the avoidance conditions.  Another consequence (not shown in Figure~\ref{propAv1233412types-structure-1} but shown in subsequent diagrams) is that every point in $\D_3$ is to the right of every point in $\D_1$, and every point in $\D_2$ is to the right of every point in $\D_4$.   If we choose the copy of $3412$ so that the `$4$' and the `$1$' points are as close (vertically) as possible then the two cells labeled $Z$ must be empty and the cell labeled $Y$ must be decreasing.

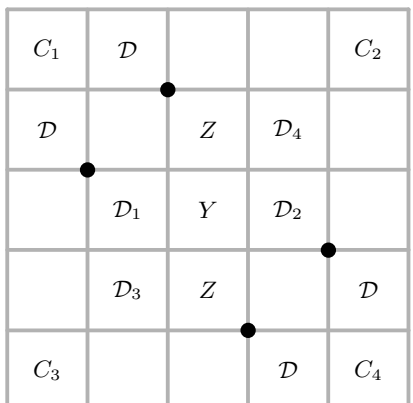
\begin{figure}
\begin{footnotesize}
\begin{center}
\psset{xunit=0.007in, yunit=0.007in}
\psset{linewidth=0.005in}
\begin{pspicture}(0,0)(300,300)
\multirput(0,0)(60,0){6}{\psline[linecolor=darkgray,linestyle=solid,linewidth=0.02in]{c-c}(0,0)(0,300)}
\multirput(0,0)(0,60){6}{\psline[linecolor=darkgray,linestyle=solid,linewidth=0.02in]{c-c}(0,0)(300,0)}
\pscircle*(60,180){0.04in}
\pscircle*(120,240){0.04in}
\pscircle*(180,60){0.04in}
\pscircle*(240,120){0.04in}
\rput[c](30,270){$C_1$}
\rput[c](30,210){$\D$}
\rput[c](30,30){$C_3$}
\rput[c](90,90){$\D_3$}
\rput[c](90,150){$\D_1$}
\rput[c](90,270){$\D$}
\rput[c](150,90){$Z$}
\rput[c](150,150){$Y$}
\rput[c](150,210){$Z$}
\rput[c](210,30){$\D$}
\rput[c](210,150){$\D_2$}
\rput[c](210,210){$\D_4$}
\rput[c](270,30){$C_4$}
\rput[c](270,90){$\D$}
\rput[c](270,270){$C_2$}
\end{pspicture}
\end{center}
\end{footnotesize}
\caption{The first structure diagram for the simple $3412$-containing permutation $\sigma\in\Av(4123, 2341)$.}
\label{propAv1233412types-structure-1}
\end{figure}

However rather more can be gleaned from the vertical proximity of the `$4$' and `$1$'.  Consider the two cells $\D_1$ and $\D_2$ that flank the center cell labeled $Y$.  There can be no increase from $\D_1$ to $Y$ nor from $Y$ to  $\D_2$ because any such increase would result in a copy of $3412$ with a closer `$4$' and `$1$'.  The diagram on the left of Figure~\ref{propAv1233412types-structure-23} displays these conditions.  Again in this diagram all cells that are not labeled are empty, and no claim is yet made about the four corner cells labeled $C_i$.

\begin{figure}
\begin{footnotesize}
\begin{center}
\begin{tabular}{ccc}
	\psset{xunit=0.007in, yunit=0.007in}
	\psset{linewidth=0.005in}
	\begin{pspicture}(0,0)(300,300)
	\multirput(0,0)(60,0){6}{\psline[linecolor=darkgray,linestyle=solid,linewidth=0.02in]{c-c}(0,0)(0,300)}
	\multirput(0,0)(0,60){6}{\psline[linecolor=darkgray,linestyle=solid,linewidth=0.02in]{c-c}(0,0)(300,0)}
	\pscircle*(60,180){0.04in}
	\pscircle*(120,240){0.04in}
	\pscircle*(180,60){0.04in}
	\pscircle*(240,120){0.04in}
	\psline[linecolor=darkgray,linestyle=solid,linewidth=0.02in]{c-c}(90,0)(90,180)
	\psline[linecolor=darkgray,linestyle=solid,linewidth=0.02in]{c-c}(210,120)(210,240)
	\psline[linecolor=darkgray,linestyle=solid,linewidth=0.02in]{c-c}(60,140)(300,140)
	\psline[linecolor=darkgray,linestyle=solid,linewidth=0.02in]{c-c}(60,160)(300,160)
	\rput[c](75,170){$\D_1$}
	\rput[c](105,90){$\D_3$}
	\rput[c](150,150){$\D$}
	\rput[c](195,210){$\D_4$}
	\rput[c](225,130){$\D_2$}
	\rput[c](30,270){$C_1$}
	\rput[c](30,210){$\D$}
	\rput[c](30,30){$C_3$}
	\rput[c](90,270){$\D$}
	\rput[c](210,30){$\D$}
	\rput[c](270,30){$C_4$}
	\rput[c](270,90){$\D$}
	\rput[c](270,270){$C_2$}
	\end{pspicture}
&\rule{0.3in}{0pt}&
	\psset{xunit=0.007in, yunit=0.007in}
	\psset{linewidth=0.005in}
	\begin{pspicture}(0,0)(300,300)
	\multirput(0,0)(60,0){6}{\psline[linecolor=darkgray,linestyle=solid,linewidth=0.02in]{c-c}(0,0)(0,300)}
	\multirput(0,0)(0,60){6}{\psline[linecolor=darkgray,linestyle=solid,linewidth=0.02in]{c-c}(0,0)(300,0)}
	\pscircle*(60,180){0.04in}
	\pscircle*(120,240){0.04in}
	\pscircle*(180,60){0.04in}
	\pscircle*(240,120){0.04in}
	\psline[linecolor=darkgray,linestyle=solid,linewidth=0.02in]{c-c}(90,0)(90,180)
	\psline[linecolor=darkgray,linestyle=solid,linewidth=0.02in]{c-c}(210,120)(210,240)
	\psline[linecolor=darkgray,linestyle=solid,linewidth=0.02in]{c-c}(60,140)(300,140)
	\psline[linecolor=darkgray,linestyle=solid,linewidth=0.02in]{c-c}(60,160)(300,160)
	\rput[c](75,170){$\D_1$}
	\rput[c](105,90){$\D_3$}
	\rput[c](150,150){$\D$}
	\rput[c](195,210){$\D_4$}
	\rput[c](225,130){$\D_2$}
	\rput[c](30,270){$C_1$}
	\rput[c](30,210){$\D$}
	\rput[c](90,270){$\D$}
	\rput[c](210,30){$\D$}
	\rput[c](270,30){$C_4$}
	\rput[c](270,90){$\D$}
	\end{pspicture}
\end{tabular}
\end{center}
\end{footnotesize}
\caption{The second and third structure diagrams for the simple $3412$-containing permutation $\sigma\in\Av(4123, 2341)$.}
\label{propAv1233412types-structure-23}
\end{figure}
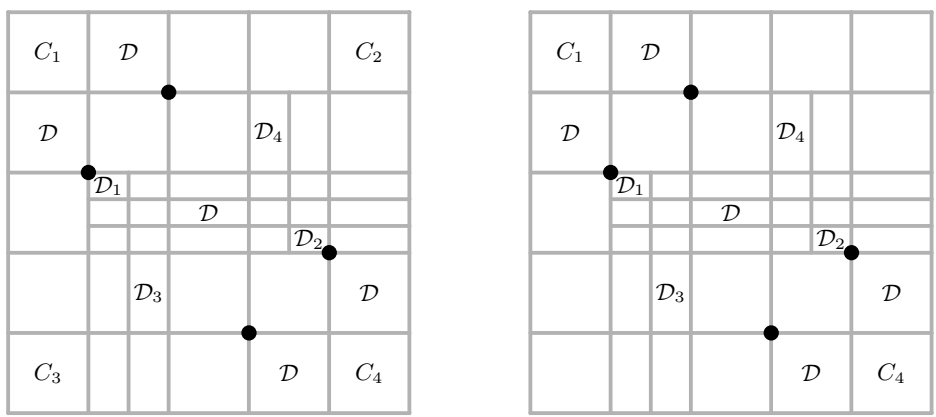

Notice that the central $\D$ cell is either empty or consists of a single point; if it were any larger it would comprise a non-trivial interval.  Now consider the cell labeled $C_1$ and the $\D$ cell to its right.  If any entry in this pair of cells was larger than any entry of $C_2$ we would have a copy $4123$.  Similarly, if any entry in the cell labeled $C_4$ or the $\D$ cell above it were to the right of an entry in the cell labeled $C_2$ then we would have a copy of $2341$.  Therefore the entries in the cell labeled $C_2$ must lie above and to the right of all other entries in $\sigma$, and so $C_2$ must be empty by simplicity.  Similarly, it can be seen that $C_3$ must be empty.  This gives the diagram on the right of Figure~\ref{propAv1233412types-structure-23}.

In this diagram the $2\times 2$ array of cells in the top-left cannot contain $123$ (since that would give a copy of $2341$).  Similarly, the $2\times 2$ array of cells in the bottom-right cannot contain $123$.  However, we have assumed that $\sigma$ contains a copy of $123$.  By inspection, this copy of $123$ must be formed by an entry in $\D_3$, an entry in the central $\D$ cell (which is known to be at most a singleton), and an entry in $\D_4$.  This, by the avoidance conditions, implies that the cells labeled $C_1$ and $C_4$ must be empty.

Now consider the cell in the top row labeled $\D$.  This cell cannot contain an entry to the left of an entry in the $\D_3$ cell because that would create a copy of $4123$, so all of the entries in this cell must lie to the right of all of the entries in $\D_3$.  However, if there are any such entries, then they would form an interval with the `$4$' of the identified copy of $3412$, so the $\D$ cell in the top row must be empty.  Similarly, it can be seen that the three other peripheral $\D$ cells are empty.  By simplicity it then follows that the cells labeled $\D_1$ and $\D_2$ must be empty and that $\D_3$ and $\D_4$ must be singletons.  This shows that $\sigma= 5274163$, as desired.
\end{proof}

\begin{corollary}\label{simpletypes}
If $\sigma$ is a simple permutation in $\Av(2341, 4123)$ then either
\begin{enumerate}
\item $\sigma$ contains $123$ but not $3412$,
\item $\sigma$ contains $3412$ but not $123$,
\item $\sigma$ contains both $123$ and $3412$ and is the permutation $5274163$,
\item $\sigma$ contains neither $123$ nor $3412$.  There are exactly two such permutations of this type of every even length $n\ge 4$.
\end{enumerate}
\end{corollary}
\begin{proof}
The final alternative is the only one that does not follow directly from Proposition~\ref{propsimpletypes}.  However the form of permutations in $\Av(123, 3412)$ is given in Proposition~\ref{propAv1233412types}.  Such a permutation can be simple only if the two decreasing sequences shown in that proposition exactly interlace and the permutation neither  begins with its largest entry nor ends with its smallest entry.  Hence the permutation must have even length and there is one such for each lengfor each of the two forms in Proposition~\ref{propAv1233412types}.
\end{proof}

\section{The Structure of Simple Permutations in $\Av(2341, 4123, 3412)$}

We now work towards a description of the simple permutations in $\Av(2341, 4123, 3412)$.  The following result, which actually holds for all permutations in this class, is a stepping stone towards that goal.

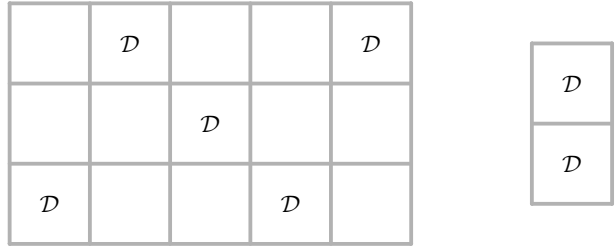
\begin{figure}
\begin{footnotesize}
\begin{center}
\begin{tabular}{ccc}
	\psset{xunit=0.007in, yunit=0.007in}
	\psset{linewidth=0.005in}
	\begin{pspicture}(0,0)(300,180)
	\multirput(0,0)(60,0){6}{\psline[linecolor=darkgray,linestyle=solid,linewidth=0.02in]{c-c}(0,0)(0,180)}
	\multirput(0,0)(0,60){4}{\psline[linecolor=darkgray,linestyle=solid,linewidth=0.02in]{c-c}(0,0)(300,0)}
	\rput[c](30,30){$\D$}
	\rput[c](90,150){$\D$}
	\rput[c](150,90){$\D$}
	\rput[c](210,30){$\D$}
	\rput[c](270,150){$\D$}
	\end{pspicture}
&\rule{0.3in}{0pt}&
	\psset{xunit=0.007in, yunit=0.007in}
	\psset{linewidth=0.005in}
	\begin{pspicture}(0,-30)(60,120)
	\multirput(0,0)(60,0){2}{\psline[linecolor=darkgray,linestyle=solid,linewidth=0.02in]{c-c}(0,0)(0,120)}
	\multirput(0,0)(0,60){3}{\psline[linecolor=darkgray,linestyle=solid,linewidth=0.02in]{c-c}(0,0)(60,0)}
	\rput[c](30,30){$\D$}
	\rput[c](30,90){$\D$}
	\end{pspicture}
\end{tabular}
\end{center}
\end{footnotesize}
\caption{The two types of permutation in Proposition~\ref{nbefore1}.  As usual, cells labeled by $\D$ represent decreasing sequences.}
\label{twotypesfigure}
\end{figure}

\begin{proposition}\label{nbefore1}
Every permutation $\pi\in\Av(2341, 4123, 3412)$ in which the greatest entry precedes the least entry has one of the two forms shown in Figure~\ref{twotypesfigure}.
\end{proposition}

\begin{proof}
Consider an arbitrary permutation $\pi\in\Av(2341,4123, 3412)$ of length $n$ in which $n$ precedes $1$.  If either $n$ is the first entry or $1$ is the final entry, then the remainder of $\pi$ avoids $123$ and $3412$, and the result follows from Proposition~\ref{propAv1233412types}.  Suppose now that $n$ is not the first entry of $\pi$ and that $1$ is not the last entry of $\pi$.  Because $\pi$ avoids $3412$, this implies that $\pi(1)<\pi(n)$, giving the situation depicted on the left of Figure~\ref{fig1}.

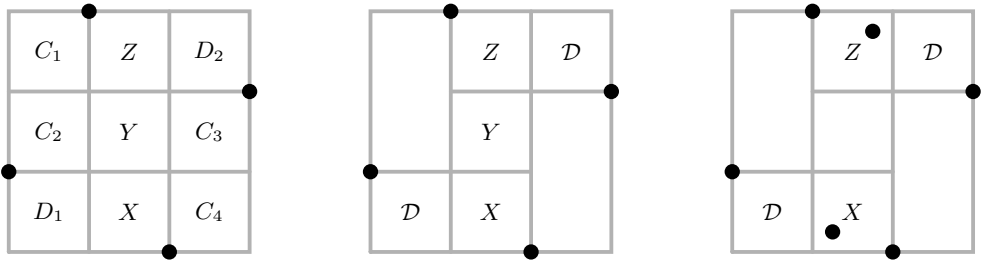
\begin{figure}
\begin{footnotesize}
\begin{center}
\begin{tabular}{ccccc}
	\psset{xunit=0.007in, yunit=0.007in}
	\psset{linewidth=0.005in}
	\begin{pspicture}(0,0)(180,180)
	\psline[linecolor=darkgray,linestyle=solid,linewidth=0.02in]{c-c}(0,0)(180,0)(180,180)(0,180)(0,0)
	\psline[linecolor=darkgray,linestyle=solid,linewidth=0.02in]{c-c}(0,60)(180,60)
	\psline[linecolor=darkgray,linestyle=solid,linewidth=0.02in]{c-c}(0,120)(180,120)
	\psline[linecolor=darkgray,linestyle=solid,linewidth=0.02in]{c-c}(60,0)(60,180)
	\psline[linecolor=darkgray,linestyle=solid,linewidth=0.02in]{c-c}(120,0)(120,180)
	\pscircle*(0,60){0.04in}
	\pscircle*(60,180){0.04in}
	\pscircle*(120,0){0.04in}
	\pscircle*(180,120){0.04in}
	\rput[c](30,30){$D_1$}
	\rput[c](90,30){$X$}
	\rput[c](90,90){$Y$}
	\rput[c](90,150){$Z$}
	\rput[c](150,150){$D_2$}
	\rput[c](30,90){$C_2$}
	\rput[c](30,150){$C_1$}
	\rput[c](150,90){$C_3$}
	\rput[c](150,30){$C_4$}
	\end{pspicture}
&\rule{0.3in}{0pt}&
	\psset{xunit=0.007in, yunit=0.007in}
	\psset{linewidth=0.005in}
	\begin{pspicture}(0,0)(180,180)
	\psline[linecolor=darkgray,linestyle=solid,linewidth=0.02in]{c-c}(0,0)(180,0)(180,180)(0,180)(0,0)
	\psline[linecolor=darkgray,linestyle=solid,linewidth=0.02in]{c-c}(0,60)(120,60)
	\psline[linecolor=darkgray,linestyle=solid,linewidth=0.02in]{c-c}(60,120)(180,120)
	\psline[linecolor=darkgray,linestyle=solid,linewidth=0.02in]{c-c}(60,0)(60,180)
	\psline[linecolor=darkgray,linestyle=solid,linewidth=0.02in]{c-c}(120,0)(120,180)
	\pscircle*(0,60){0.04in}
	\pscircle*(60,180){0.04in}
	\pscircle*(120,0){0.04in}
	\pscircle*(180,120){0.04in}
	\rput[c](30,30){$\D$}
	\rput[c](90,30){$X$}
	\rput[c](90,90){$Y$}
	\rput[c](90,150){$Z$}
	\rput[c](150,150){$\D$}
	\end{pspicture}
&\rule{0.3in}{0pt}&
	\psset{xunit=0.007in, yunit=0.007in}
	\psset{linewidth=0.005in}
	\begin{pspicture}(0,0)(180,180)
	\psline[linecolor=darkgray,linestyle=solid,linewidth=0.02in]{c-c}(0,0)(180,0)(180,180)(0,180)(0,0)
	\psline[linecolor=darkgray,linestyle=solid,linewidth=0.02in]{c-c}(0,60)(120,60)
	\psline[linecolor=darkgray,linestyle=solid,linewidth=0.02in]{c-c}(60,120)(180,120)
	\psline[linecolor=darkgray,linestyle=solid,linewidth=0.02in]{c-c}(60,0)(60,180)
	\psline[linecolor=darkgray,linestyle=solid,linewidth=0.02in]{c-c}(120,0)(120,180)
	\pscircle*(0,60){0.04in}
	\pscircle*(60,180){0.04in}
	\pscircle*(120,0){0.04in}
	\pscircle*(180,120){0.04in}
	\rput[c](30,30){$\D$}
	\rput[c](90,30){$X$}
	\rput[c](90,150){$Z$}
	\rput[c](150,150){$\D$}
	\pscircle*(75,15){0.04in}
	\pscircle*(105,165){0.04in}
	\end{pspicture}
\end{tabular}
\end{center}
\end{footnotesize}
\caption{Structure diagrams for a permutation in $\Av(2341,4123,3412)$ in which $n$ precedes $1$ but $n$ is not the first entry and $1$ is not the last entry.}
\label{fig1}
\end{figure}

From the fact that $\pi$ avoids $2341$, we see that the cells labeled $C_1$ and $C_2$ must be empty, while the cell $D_1$ must be decreasing.  Using the $4123$ avoidance of $\pi$, we see that the cells labeled $C_3$ and $C_4$ must be empty, while the cell $D_2$ must be decreasing.  This gives the center diagram of Figure~\ref{fig1}.

The 2341-avoidance proves that the region $Y\cup Z$ is decreasing, while the 4123-avoidance proves that $X\cup Y$ is decreasing.  If the entire region $X\cup Y\cup Z$ is decreasing then $\pi$ has the structure shown on the left of Figure~\ref{twotypesfigure}, and we are done.

So suppose to the contrary that $X\cup Y\cup Z$ is not decreasing.  Then cell $Y$ must be empty.  Furthermore, the first (and largest) point of cell $X$ must precede the last (and smallest) point of cell $Z$ and it follows from the $2341$, $4123$-avoidance again that the cell $X$ and the cell labeled $\D$ to its left must form a single decreasing sequence, as must the cell $Z$ and the cell labeled $\D$ to its right.  The permutation $\pi$ is therefore a vertical juxtaposition of decreasing sequences, which is the structure on the right of Figure~\ref{twotypesfigure}, completing the proof.
%
\end{proof}

\begin{corollary}\label{firstgtrlast}
Every permutation of length $n$ in $\Av(2341, 4123, 3412)$ whose first entry is greater than its last entry has one of the forms of Figure~\ref{twotypesfigureinverse}.
\end{corollary}
\begin{proof}
The permutations of the corollary are the inverses of the permutations of Proposition~\ref{nbefore1}.  The result follows because inversion is represented by reflection of permutation diagrams about the southwest-northeast diagonal and $\Av(2341, 4123, 3412)$ is closed under taking inverses.
\end{proof}

%
%
%
%

\begin{figure}
\begin{footnotesize}
\begin{center}
\begin{tabular}{ccc}
	\psset{xunit=0.007in, yunit=0.007in}
	\psset{linewidth=0.005in}
	\begin{pspicture}(0,0)(180,300)
	\multirput(0,0)(60,0){4}{\psline[linecolor=darkgray,linestyle=solid,linewidth=0.02in]{c-c}(0,0)(0,300)}
	\multirput(0,0)(0,60){6}{\psline[linecolor=darkgray,linestyle=solid,linewidth=0.02in]{c-c}(0,0)(180,0)}
	\rput[c](30,30){$\D$}
	\rput[c](30,210){$\D$}
	\rput[c](90,150){$\D$}
	\rput[c](150,90){$\D$}
	\rput[c](150,270){$\D$}
	\end{pspicture}
&\rule{0.3in}{0pt}&
	\psset{xunit=0.007in, yunit=0.007in}
	\psset{linewidth=0.005in}
	\begin{pspicture}(0,-120)(120,60)
	\multirput(0,0)(60,0){3}{\psline[linecolor=darkgray,linestyle=solid,linewidth=0.02in]{c-c}(0,0)(0,60)}
	\multirput(0,0)(0,60){2}{\psline[linecolor=darkgray,linestyle=solid,linewidth=0.02in]{c-c}(0,0)(120,0)}
	\rput[c](30,30){$\D$}
	\rput[c](90,30){$\D$}
	\end{pspicture}
\end{tabular}
\end{center}
\end{footnotesize}
\caption{The two types of permutation in Corollary~\ref{firstgtrlast}.}
\label{twotypesfigureinverse}
\end{figure}
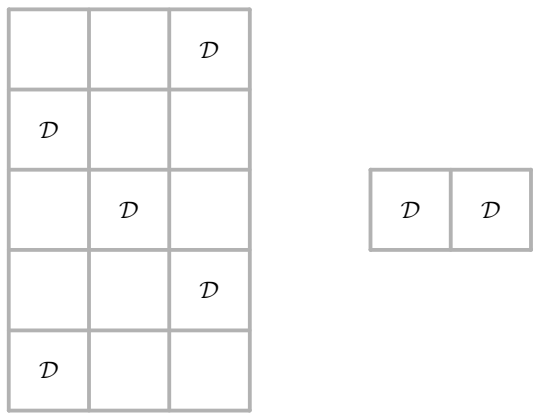

The previous two results have described very restricted subsets of $\Av(2341, 4123, 3412)$.  We now broaden our study to consider arbitrary permutations in this class.  Recall that the entry $\pi(j)$ of $\pi$ is a \emph{left-to-right maxima} (\emph{l-r max} for short) if $\pi(j)>\pi(i)$ for all $i<j$, and a \emph{right-to-left minima} (\emph{r-l min} for short) if $\pi(j)<\pi(k)$ for all $k>j$.  It is convenient to connect the l-r maxes and connect the r-l mins by axes-parallel paths as depicted in Figure~\ref{MinsMaxs}.  In this \emph{relative extrema diagram} the entries of the permutation are depicted by circles as usual while the squares denote \emph{inflections} in the axes-parallel paths.  From the definition of l-r maxes and r-l mins, it follows that there are no entries above the l-r max path and no entries below the r-l min path.

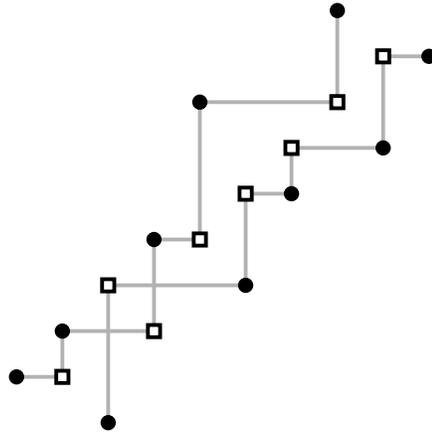
\begin{figure}
\begin{footnotesize}
\begin{center}
	\psset{xunit=0.012in, yunit=0.012in}
	\psset{linewidth=0.005in}
	\begin{pspicture}(0,0)(180,180)
	\psline[linecolor=darkgray,linestyle=solid,linewidth=0.02in]{c-c}(0,20)(20,20)(20,40)(60,40)(60,80)(80,80)(80,140)(140,140)(140,180)
	\pscircle*(0,20){0.04in}
	\rput[c](20,20){\psframe[linewidth=0.02in,fillstyle=solid,fillcolor=white](-3.5,-3.5)(3.5,3.5)}
	\pscircle*(20,40){0.04in}
	\rput[c](60,40){\psframe[linewidth=0.02in,fillstyle=solid,fillcolor=white](-3.5,-3.5)(3.5,3.5)}
	\pscircle*(60,80){0.04in}
	\rput[c](80,80){\psframe[linewidth=0.02in,fillstyle=solid,fillcolor=white](-3.5,-3.5)(3.5,3.5)}
	\pscircle*(80,140){0.04in}
	\rput[c](140,140){\psframe[linewidth=0.02in,fillstyle=solid,fillcolor=white](-3.5,-3.5)(3.5,3.5)}
	\pscircle*(140,180){0.04in}
	\psline[linecolor=darkgray,linestyle=solid,linewidth=0.02in]{c-c}(40,0)(40,60)(100,60)(100,100)(120,100)(120,120)(160,120)(160,160)(180,160)
	\pscircle*(40,0){0.04in}
	\rput[c](40,60){\psframe[linewidth=0.02in,fillstyle=solid,fillcolor=white](-3.5,-3.5)(3.5,3.5)}
	\pscircle*(100,60){0.04in}
	\rput[c](100,100){\psframe[linewidth=0.02in,fillstyle=solid,fillcolor=white](-3.5,-3.5)(3.5,3.5)}
	\pscircle*(120,100){0.04in}
	\rput[c](120,120){\psframe[linewidth=0.02in,fillstyle=solid,fillcolor=white](-3.5,-3.5)(3.5,3.5)}
	\pscircle*(160,120){0.04in}
	\rput[c](160,160){\psframe[linewidth=0.02in,fillstyle=solid,fillcolor=white](-3.5,-3.5)(3.5,3.5)}
	\pscircle*(180,160){0.04in}
	\end{pspicture}
\end{center}
\end{footnotesize}
\caption{L-r maxes and r-l mins in an arbitary permutation, connected by axes-parallel paths.}
\label{MinsMaxs}
\end{figure}

In our situation there are strong conditions on the interaction between the l-r maxes and the r-l mins.

\begin{lemma}\label{alternating}
If $\pi\in\Av(2341, 4123, 3412)$ is sum indecomposable then the inflection points form an increasing sequence in which the inflections associated with the l-r maxes alternate with the inflections associated with the r-l mins.
\end{lemma}

\begin{proof}
The inflections associated with the l-r maxes are all increasing by definition, as are the inflections associated with the r-l mins.  To show that their union is increasing we just have to show that neither of the two situations on the left of Figure~\ref{FourIllegals} can arise.

\begin{figure}
\begin{footnotesize}
\begin{center}
\begin{tabular}{ccccccc}
	\psset{xunit=0.012in, yunit=0.012in}
	\psset{linewidth=0.005in}
	\begin{pspicture}(0,-20)(60,60)
	\psline[linecolor=darkgray,linestyle=solid,linewidth=0.02in]{c-c}(0,40)(20,40)(20,60)
	\pscircle*(0,40){0.04in}
	\rput[c](20,40){\psframe[linewidth=0.02in,fillstyle=solid,fillcolor=white](-3.5,-3.5)(3.5,3.5)}
	\pscircle*(20,60){0.04in}
	\psline[linecolor=darkgray,linestyle=solid,linewidth=0.02in]{c-c}(40,0)(40,20)(60,20)
	\pscircle*(40,0){0.04in}
	\rput[c](40,20){\psframe[linewidth=0.02in,fillstyle=solid,fillcolor=white](-3.5,-3.5)(3.5,3.5)}
	\pscircle*(60,20){0.04in}
	\end{pspicture}
&\rule{0.3in}{0pt}&
	\psset{xunit=0.012in, yunit=0.012in}
	\psset{linewidth=0.005in}
	\begin{pspicture}(0,-20)(60,60)
	\psline[linecolor=darkgray,linestyle=solid,linewidth=0.02in]{c-c}(0,20)(40,20)(40,60)
	\pscircle*(0,20){0.04in}
	\rput[c](40,20){\psframe[linewidth=0.02in,fillstyle=solid,fillcolor=white](-3.5,-3.5)(3.5,3.5)}
	\pscircle*(40,60){0.04in}
	\psline[linecolor=darkgray,linestyle=solid,linewidth=0.02in]{c-c}(20,0)(20,40)(60,40)
	\pscircle*(20,0){0.04in}
	\rput[c](20,40){\psframe[linewidth=0.02in,fillstyle=solid,fillcolor=white](-3.5,-3.5)(3.5,3.5)}
	\pscircle*(60,40){0.04in}
	\end{pspicture}
&\rule{0.3in}{0pt}&
	\psset{xunit=0.012in, yunit=0.012in}
	\psset{linewidth=0.005in}
	\begin{pspicture}(0,0)(100,100)
	\psline[linecolor=darkgray,linestyle=solid,linewidth=0.02in]{c-c}(0,40)(40,40)(40,60)(60,60)(60,100)
	\pscircle*(0,40){0.04in}
	\rput[c](40,40){\psframe[linewidth=0.02in,fillstyle=solid,fillcolor=white](-3.5,-3.5)(3.5,3.5)}
	\pscircle*(40,60){0.04in}
	\rput[c](60,60){\psframe[linewidth=0.02in,fillstyle=solid,fillcolor=white](-3.5,-3.5)(3.5,3.5)}
	\pscircle*(60,100){0.04in}
	\psline[linecolor=darkgray,linestyle=solid,linewidth=0.02in]{c-c}(20,0)(20,20)(80,20)(80,80)(100,80)
	\pscircle*(20,0){0.04in}
	\rput[c](20,20){\psframe[linewidth=0.02in,fillstyle=solid,fillcolor=white](-3.5,-3.5)(3.5,3.5)}
	\pscircle*(80,20){0.04in}
	\rput[c](80,80){\psframe[linewidth=0.02in,fillstyle=solid,fillcolor=white](-3.5,-3.5)(3.5,3.5)}
	\pscircle*(100,80){0.04in}
	\end{pspicture}
&\rule{0.3in}{0pt}&
	\psset{xunit=0.012in, yunit=0.012in}
	\psset{linewidth=0.005in}
	\begin{pspicture}(0,0)(100,100)
	\psline[linecolor=darkgray,linestyle=solid,linewidth=0.02in]{c-c}(0,20)(20,20)(20,80)(80,80)(80,100)
	\pscircle*(0,20){0.04in}
	\rput[c](20,20){\psframe[linewidth=0.02in,fillstyle=solid,fillcolor=white](-3.5,-3.5)(3.5,3.5)}
	\pscircle*(20,80){0.04in}
	\rput[c](80,80){\psframe[linewidth=0.02in,fillstyle=solid,fillcolor=white](-3.5,-3.5)(3.5,3.5)}
	\pscircle*(80,100){0.04in}
	\psline[linecolor=darkgray,linestyle=solid,linewidth=0.02in]{c-c}(40,0)(40,40)(60,40)(60,60)(100,60)
	\pscircle*(40,0){0.04in}
	\rput[c](40,40){\psframe[linewidth=0.02in,fillstyle=solid,fillcolor=white](-3.5,-3.5)(3.5,3.5)}
	\pscircle*(60,40){0.04in}
	\rput[c](60,60){\psframe[linewidth=0.02in,fillstyle=solid,fillcolor=white](-3.5,-3.5)(3.5,3.5)}
	\pscircle*(100,60){0.04in}
	\end{pspicture}
\end{tabular}
\end{center}
\end{footnotesize}
\caption{Four illegal configurations for a permutation in $\Av(2341, 4123, 3412)$.}
\label{FourIllegals}
\end{figure}
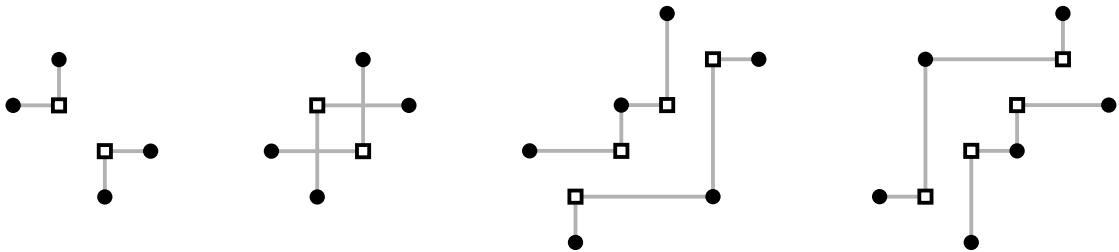

The first situation in Figure~\ref{FourIllegals} cannot arise, because it contains $3412$.  Now suppose that the second situation of Figure~\ref{FourIllegals} occurs in one of these permutations.  The inflection in the lower-right position comes from the l-r max path, so $\pi$ cannot contain any entries above this path.  Similarly, $\pi$ cannot contain any entries below the r-l min path.  This implies that $\pi$ must be sum decomposable, which is a contradiction.

It remains to show that the two types of inflection alternate.  Essentially, the only way this property can fail is if the permutation contains the third or fourth situation from Figure~\ref{FourIllegals}.  Both of these situations contain a copy of either $2341$ or $4123$, completing the proof.
\end{proof}

This lemma holds, of course, for all simple permutations in $\Av(2341,4123,3412)$ (of length more than 2) and we now build on it to pin down the structure of such permutations.
Figure~\ref{SimpleExample} shows one of the two ways in which the l-r maxes can interact with the r-l mins in such a simple permutation.  In this figure the leftmost inflection is associated with the l-r maxes; the other way is where the leftmost inflection is associated with the r-l mins and the two types are related by inversion.

Figure~\ref{SimpleExample} shows the permutation partitioned into cells: these cells are called \emph{corner} cells if they abut a l-r max or a r-l min, and \emph{central} cells otherwise.  Successive corner cells, except for the first two and final two, are always separated by a central cell.  As we shall soon see there are strong dependencies between consecutive cells. 

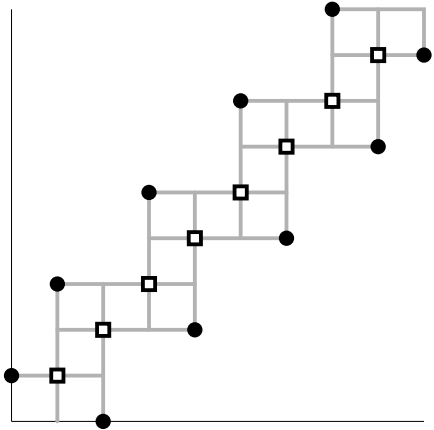
\begin{figure}
\begin{footnotesize}
\begin{center}
	\psset{xunit=0.012in, yunit=0.012in}
	\psset{linewidth=0.005in}
	\begin{pspicture}(0,0)(180,180)
	\psaxes[dy=200,dx=200,tickstyle=bottom,showorigin=false,labels=none](0,0)(180,180)
	\psline[linecolor=darkgray,linestyle=solid,linewidth=0.02in]{c-c}(0,20)(20,20)(20,60)(60,60)(60,100)(100,100)(100,140)(140,140)(140,180)
	\psline[linecolor=darkgray,linestyle=solid,linewidth=0.02in]{c-c}(20,0)(20,20)(40,20)
	\psline[linecolor=darkgray,linestyle=solid,linewidth=0.02in]{c-c}(20,40)(40,40)(40,60)
	\psline[linecolor=darkgray,linestyle=solid,linewidth=0.02in]{c-c}(60,40)(60,60)(80,60)
	\psline[linecolor=darkgray,linestyle=solid,linewidth=0.02in]{c-c}(60,80)(80,80)(80,100)
	\psline[linecolor=darkgray,linestyle=solid,linewidth=0.02in]{c-c}(100,80)(100,100)(120,100)
	\psline[linecolor=darkgray,linestyle=solid,linewidth=0.02in]{c-c}(100,120)(120,120)(120,140)
	\psline[linecolor=darkgray,linestyle=solid,linewidth=0.02in]{c-c}(140,120)(140,140)(160,140)
	\psline[linecolor=darkgray,linestyle=solid,linewidth=0.02in]{c-c}(140,160)(160,160)(160,180)
	\psline[linecolor=darkgray,linestyle=solid,linewidth=0.02in]{c-c}(140,180)(180,180)(180,160)
	\pscircle*(0,20){0.04in}
	\rput[c](20,20){\psframe[linewidth=0.02in,fillstyle=solid,fillcolor=white](-3.5,-3.5)(3.5,3.5)}
	\pscircle*(20,60){0.04in}
	\rput[c](60,60){\psframe[linewidth=0.02in,fillstyle=solid,fillcolor=white](-3.5,-3.5)(3.5,3.5)}
	\pscircle*(60,100){0.04in}
	\rput[c](100,100){\psframe[linewidth=0.02in,fillstyle=solid,fillcolor=white](-3.5,-3.5)(3.5,3.5)}
	\pscircle*(100,140){0.04in}
	\rput[c](140,140){\psframe[linewidth=0.02in,fillstyle=solid,fillcolor=white](-3.5,-3.5)(3.5,3.5)}
	\pscircle*(140,180){0.04in}
	\psline[linecolor=darkgray,linestyle=solid,linewidth=0.02in]{c-c}(40,0)(40,40)(80,40)(80,80)(120,80)(120,120)(160,120)(160,160)(180,160)
	\pscircle*(40,0){0.04in}
	\rput[c](40,40){\psframe[linewidth=0.02in,fillstyle=solid,fillcolor=white](-3.5,-3.5)(3.5,3.5)}
	\pscircle*(80,40){0.04in}
	\rput[c](80,80){\psframe[linewidth=0.02in,fillstyle=solid,fillcolor=white](-3.5,-3.5)(3.5,3.5)}
	\pscircle*(120,80){0.04in}
	\rput[c](120,120){\psframe[linewidth=0.02in,fillstyle=solid,fillcolor=white](-3.5,-3.5)(3.5,3.5)}
	\pscircle*(160,120){0.04in}
	\rput[c](160,160){\psframe[linewidth=0.02in,fillstyle=solid,fillcolor=white](-3.5,-3.5)(3.5,3.5)}
	\pscircle*(180,160){0.04in}
	\end{pspicture}
\end{center}
\end{footnotesize}
\caption{One of two possible interactions of l-r maxes and r-l mins in a permutation whose inflection points satisfy the conditions of Lemma \ref{alternating}.}
\label{SimpleExample}
\end{figure}

In such a simple permutation, consider any four points consisting of two consecutive l-r maxes and two consecutive r-l mins whose associated inflection points are also consecutive, together with the set of points of the permutation contained within the rectangle they define.  There are two possible ways for the two l-r maxes to interleave with the two r-l mins (as a $2413$ or as a $3142$), and in each case the subpermutation in the rectangle they define is of one of the types considered in Proposition~\ref{nbefore1} or Corollary~\ref{firstgtrlast}.  Therefore the possible forms for the subpermutation are as shown in Figure~\ref{4pointsandrectcontent}.  Following our conventions, the empty regions in these diagrams are empty.  Also, because of simplicity, a central cell labeled $\D$ must be either empty or a singleton.

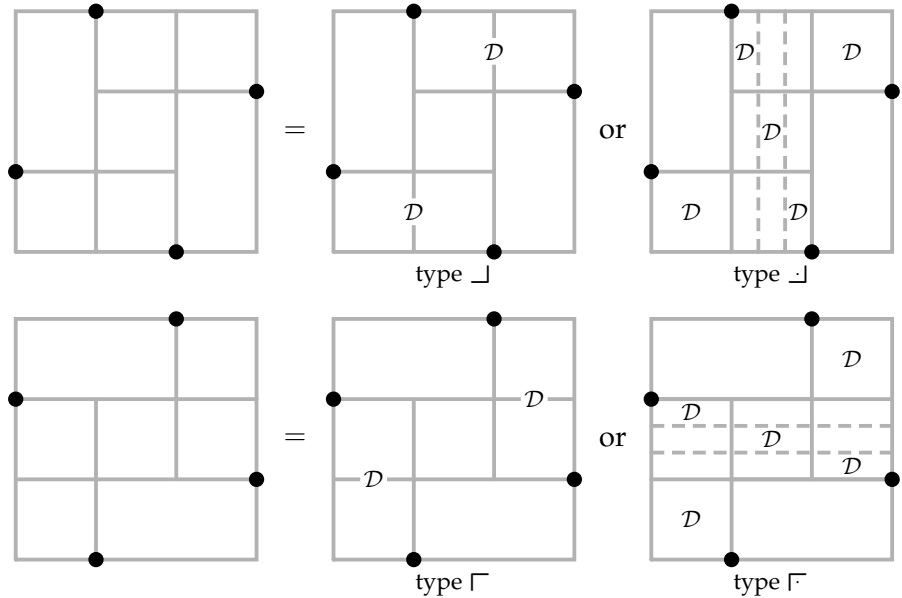
\begin{figure}
\begin{footnotesize}
\begin{center}
\begin{tabular}{ccccc}
	\psset{xunit=0.007in, yunit=0.007in}
	\psset{linewidth=0.005in}
	\begin{pspicture}(0,0)(180,180)
	\psline[linecolor=darkgray,linestyle=solid,linewidth=0.02in]{c-c}(0,0)(180,0)(180,180)(0,180)(0,0)
	\psline[linecolor=darkgray,linestyle=solid,linewidth=0.02in]{c-c}(0,60)(60,60)(60,180)
	\psline[linecolor=darkgray,linestyle=solid,linewidth=0.02in]{c-c}(120,0)(120,120)(180,120)
	\psline[linecolor=darkgray,linestyle=solid,linewidth=0.02in]{c-c}(60,0)(60,60)(120,60)
	\psline[linecolor=darkgray,linestyle=solid,linewidth=0.02in]{c-c}(60,120)(120,120)(120,180)
	\pscircle*(0,60){0.04in}
	\pscircle*(60,180){0.04in}
	\pscircle*(120,0){0.04in}
	\pscircle*(180,120){0.04in}
	\end{pspicture}
&
\psset{xunit=0.007in, yunit=0.007in}
\psset{linewidth=0.005in}
\begin{pspicture}(0,0)(10,180)
\rput[c](5,90){\begin{normalsize}$=$\end{normalsize}}
\end{pspicture}
&
	\psset{xunit=0.007in, yunit=0.007in}
	\psset{linewidth=0.005in}
	\begin{pspicture}(0,0)(180,180)
	\psline[linecolor=darkgray,linestyle=solid,linewidth=0.02in]{c-c}(0,0)(180,0)(180,180)(0,180)(0,0)
	\psline[linecolor=darkgray,linestyle=solid,linewidth=0.02in]{c-c}(0,60)(60,60)(60,180)
	\psline[linecolor=darkgray,linestyle=solid,linewidth=0.02in]{c-c}(120,0)(120,120)(180,120)
	\psline[linecolor=darkgray,linestyle=solid,linewidth=0.02in]{c-c}(60,60)(120,60)
	\psline[linecolor=darkgray,linestyle=solid,linewidth=0.02in]{c-c}(60,120)(120,120)
	\psline[linecolor=darkgray,linestyle=solid,linewidth=0.02in](60,0)(60,20)
	\psline[linecolor=darkgray,linestyle=solid,linewidth=0.02in](60,40)(60,60)
	\rput[c](60,30){$\D$}
	\psline[linecolor=darkgray,linestyle=solid,linewidth=0.02in](120,120)(120,140)
	\psline[linecolor=darkgray,linestyle=solid,linewidth=0.02in](120,160)(120,180)
	\rput[c](120,150){$\D$}
	\pscircle*(0,60){0.04in}
	\pscircle*(60,180){0.04in}
	\pscircle*(120,0){0.04in}
	\pscircle*(180,120){0.04in}
	\end{pspicture}
&
\psset{xunit=0.007in, yunit=0.007in}
\psset{linewidth=0.005in}
\begin{pspicture}(0,0)(10,180)
\rput[c](5,90){\begin{normalsize}or\end{normalsize}}
\end{pspicture}
&
	\psset{xunit=0.007in, yunit=0.007in}
	\psset{linewidth=0.005in}
	\begin{pspicture}(0,0)(180,180)
	\psline[linecolor=darkgray,linestyle=solid,linewidth=0.02in]{c-c}(0,0)(180,0)(180,180)(0,180)(0,0)
	\psline[linecolor=darkgray,linestyle=solid,linewidth=0.02in]{c-c}(0,60)(60,60)(60,180)
	\psline[linecolor=darkgray,linestyle=solid,linewidth=0.02in]{c-c}(120,0)(120,120)(180,120)
	\psline[linecolor=darkgray,linestyle=solid,linewidth=0.02in]{c-c}(60,60)(120,60)
	\psline[linecolor=darkgray,linestyle=solid,linewidth=0.02in]{c-c}(60,120)(120,120)
	\psline[linecolor=darkgray,linestyle=solid,linewidth=0.02in]{c-c}(60,0)(60,180)
	\psline[linecolor=darkgray,linestyle=dashed,linewidth=0.02in]{c-c}(80,0)(80,180)
	\psline[linecolor=darkgray,linestyle=dashed,linewidth=0.02in]{c-c}(100,0)(100,180)
	\psline[linecolor=darkgray,linestyle=solid,linewidth=0.02in]{c-c}(120,0)(120,180)
	\rput[c](30,30){$\D$}
	\rput[c](150,150){$\D$}
	\rput[c](70,150){$\D$}
	\rput[c](90,90){$\D$}
	\rput[c](110,30){$\D$}
	\pscircle*(0,60){0.04in}
	\pscircle*(60,180){0.04in}
	\pscircle*(120,0){0.04in}
	\pscircle*(180,120){0.04in}
	\end{pspicture}
\\
&&type \tilenla&&type \tilenlb
\\\\
	\psset{xunit=0.007in, yunit=0.007in}
	\psset{linewidth=0.005in}
	\begin{pspicture}(0,0)(180,180)
	\psline[linecolor=darkgray,linestyle=solid,linewidth=0.02in]{c-c}(0,0)(180,0)(180,180)(0,180)(0,0)
	\psline[linecolor=darkgray,linestyle=solid,linewidth=0.02in]{c-c}(0,120)(120,120)(120,180)
	\psline[linecolor=darkgray,linestyle=solid,linewidth=0.02in]{c-c}(60,0)(60,60)(180,60)
	\psline[linecolor=darkgray,linestyle=solid,linewidth=0.02in]{c-c}(0,60)(60,60)(60,120)
	\psline[linecolor=darkgray,linestyle=solid,linewidth=0.02in]{c-c}(120,60)(120,120)(180,120)
	\pscircle*(0,120){0.04in}
	\pscircle*(60,0){0.04in}
	\pscircle*(120,180){0.04in}
	\pscircle*(180,60){0.04in}
	\end{pspicture}
&
\psset{xunit=0.007in, yunit=0.007in}
\psset{linewidth=0.005in}
\begin{pspicture}(0,0)(10,180)
\rput[c](5,90){\begin{normalsize}$=$\end{normalsize}}
\end{pspicture}
&
	\psset{xunit=0.007in, yunit=0.007in}
	\psset{linewidth=0.005in}
	\begin{pspicture}(0,0)(180,180)
	\psline[linecolor=darkgray,linestyle=solid,linewidth=0.02in]{c-c}(0,0)(180,0)(180,180)(0,180)(0,0)
	\psline[linecolor=darkgray,linestyle=solid,linewidth=0.02in]{c-c}(0,120)(120,120)(120,180)
	\psline[linecolor=darkgray,linestyle=solid,linewidth=0.02in]{c-c}(60,0)(60,60)(180,60)
	\pscircle*(0,120){0.04in}
	\pscircle*(60,0){0.04in}
	\pscircle*(120,180){0.04in}
	\pscircle*(180,60){0.04in}
	\psline[linecolor=darkgray,linestyle=solid,linewidth=0.02in]{c-c}(60,60)(60,120)
	\psline[linecolor=darkgray,linestyle=solid,linewidth=0.02in]{c-c}(120,60)(120,120)
	\psline[linecolor=darkgray,linestyle=solid,linewidth=0.02in](0,60)(20,60)
	\psline[linecolor=darkgray,linestyle=solid,linewidth=0.02in](40,60)(60,60)
	\rput[c](30,60){$\D$}
	\psline[linecolor=darkgray,linestyle=solid,linewidth=0.02in](120,120)(140,120)
	\psline[linecolor=darkgray,linestyle=solid,linewidth=0.02in](160,120)(180,120)
	\rput[c](150,120){$\D$}
	\end{pspicture}
&
\psset{xunit=0.007in, yunit=0.007in}
\psset{linewidth=0.005in}
\begin{pspicture}(0,0)(10,180)
\rput[c](5,90){\begin{normalsize}or\end{normalsize}}
\end{pspicture}
&
\psset{xunit=0.007in, yunit=0.007in}
\psset{linewidth=0.005in}
\begin{pspicture}(0,0)(180,180)
\psline[linecolor=darkgray,linestyle=solid,linewidth=0.02in]{c-c}(0,0)(180,0)(180,180)(0,180)(0,0)
	\psline[linecolor=darkgray,linestyle=solid,linewidth=0.02in]{c-c}(0,120)(120,120)(120,180)
	\psline[linecolor=darkgray,linestyle=solid,linewidth=0.02in]{c-c}(60,0)(60,60)(180,60)
\psline[linecolor=darkgray,linestyle=solid,linewidth=0.02in]{c-c}(60,60)(60,120)
\psline[linecolor=darkgray,linestyle=solid,linewidth=0.02in]{c-c}(120,60)(120,120)
\psline[linecolor=darkgray,linestyle=solid,linewidth=0.02in]{c-c}(0,60)(180,60)
\psline[linecolor=darkgray,linestyle=dashed,linewidth=0.02in]{c-c}(0,80)(180,80)
\psline[linecolor=darkgray,linestyle=dashed,linewidth=0.02in]{c-c}(0,100)(180,100)
\psline[linecolor=darkgray,linestyle=solid,linewidth=0.02in]{c-c}(0,120)(180,120)
\rput[c](30,30){$\D$}
\rput[c](150,150){$\D$}
\rput[c](30,110){$\D$}
\rput[c](90,90){$\D$}
\rput[c](150,70){$\D$}
	\pscircle*(0,120){0.04in}
	\pscircle*(60,0){0.04in}
	\pscircle*(120,180){0.04in}
	\pscircle*(180,60){0.04in}
\end{pspicture}
\\
&&type \tilelna&&type \tilelnb
\\
\end{tabular}
\end{center}
\end{footnotesize}
\caption{The possible forms of the rectangles defined by two consecutive l-r maxes and two consecutive l-r mins in a simple permutation in $\Av(2341, 4123, 3412)$.}
\label{4pointsandrectcontent}
\end{figure}

We call the rectangles of Figure~\ref{4pointsandrectcontent} formed from the l-r maxes and r-l mins of a simple permutation in $\Av(2341,4123,3412)$ the ``tiles'' of the permutation, and we say that each tile is of type $2413$ or $3142$ (the type being determined by the relative order of the four extremal entries).  The whole permutation is then a union of overlapping, alternating tiles (overlapping in strips and alternating in type).  Since we know the structure of the tiles and that they must fit together with compatible intersections we can deduce strong consequences for the cells of a permutation.

To give suggestive names to the tiles in Figure~\ref{4pointsandrectcontent} call the diagrams in the first row $\tilenla$ and $\tilenlb$ left-to-right.  Similarly call the second row diagrams $\tilelna$ and $\tilelnb$.  

\begin{theorem}\label{sAv(2341-4123-3412}
A simple permutation $\pi$ of length more than 2 lies in $\Av(2341,4123,3412)$ if and only if it satisfies the following four conditions.
\begin{enumerate}
\item[(a)] The inflection points form an increasing sequence in which the inflections associated with the l-r maxes alternate with the inflections associated with r-l mins.
\item[(b)] The corner cells of $\pi$ are decreasing.
\item[(c)] Every pair of consecutive corner cells either form a decreasing sequence or interlace as a parallel alternation in the sense of Figure~\ref{parallelalternations}
\item[(d)] A non-empty corner cell interlaces with either the previous or next corner cell, but not both.
\item[(e)]  A central cell $s$ contains at most one element.   If the two corners adjacent to it interlace then $s$ is empty; otherwise these two cells together with $s$ form a decreasing sequence.
\end{enumerate}
\end{theorem}
\begin{proof}
We first prove that every simple permutation in $\Av(2341,4123,3412)$ satisfies (a)--(e), beginning by observing that (a) is a direct consequence of Lemma~\ref{alternating}.  By (a), the graph of $\pi$ can be decomposed into tiles of one of the four forms shown in Figure~\ref{4pointsandrectcontent}.  All corners in these tiles are decreasing, so (b) must hold.

To prove (c) consider any two consecutive corner cells.  We shall assume that they are neither the first nor last pair of corner cells (these exceptional cases are treated by an almost identical argument).  The two corner cells are separated by some central cell and lie in a tile in which they are the second and third corners of the tile.   If the tile is of type \tilelnb or $\tilenlb$ then these cells form a decreasing sequence.  If the tile is of type $\tilelna$ or $\tilenla$ then they must interlace as a parallel alternation for otherwise one of the cells will contain two elements forming a block or a single element forming a block with the l-r max or r-l min on its boundary, which contradicts simplicity. 


To prove (d) let $A, B, C$ be 3 consecutive corner cells with the middle cell $B$ non-empty.  Consider a tile containing them whose first three corner cells correspond to $A, B, C$ (the case where the last three corner cells of the tile correspond to $A, B, C$ is similar).  If this tile is of  type $\tilelna$ or $\tilenla$ then $A\cup B$ is decreasing while, by the argument of (c), $B$ and $C$ interlace.  If the tile is of type $\tilelnb$ or $\tilenlb$ then $B\cup C$ is certainly decreasing.  But, if $A\cup B$ is also decreasing then $B$, together with its abutting extremal point, would be a non-trivial interval of $\pi$ contradicting simplicity.

For the proof of (e) note first that every central cell $s$ of $\pi$ is the central cell of some tile and we have already observed that such central cells have at most one point.  
If the adjacent corner cells of this tile interlace then the tile has type $\tilelna$ or $\tilenla$ and so $s$ is empty.  Otherwise the adjacent corners form a decreasing sequence and (if this is non-empty) the tile has type $\tilelnb$ or $\tilenlb$ in which case the corners also form, together with $s$, a decreasing sequence.


For the converse let $\pi$ be any permutation satisfying (a)--(e).  Condition (a) shows that the l-r maxes and r-l mins of $\pi$ or $\pi^{-1}$ interlace as shown in Figure \ref {SimpleExample} and conditions (b)-- (e) show that the permutation is an overlapping union of the tiles shown in Figure \ref {4pointsandrectcontent}. 

If $\pi$ contains a copy of 2341 then, by replacing the `1' in this copy by a subsequent smaller point if necessary, we may take the `1' to be a r-l min.  Then (see Figure \ref{SimpleExample}) the points corresponding to the `2', `3', and `4' must lie in the 4 or 5 cells that contain points before and larger than the `1'.  In particular a copy of 2341 is contained in a tile and, from the form of the tiles, this is impossible.  Thus $\pi$ avoids 2341 and, by a similar argument also avoids 4123.

Suppose now that $\pi$ contains a copy of  3412.  The `3' in this copy must be contained in a cell associated with a l-r max because points in other types of cell are never followed by a smaller increasing pair of points.  The `4' in this copy is not contained in the same cell as the `3' nor in the immediately succeeding central cell (since, from the form of the tiles, the entries in adjacent cells form a decreasing sequence).  Hence the only location where the `1' and `2' can be situated is in the corner cell following the one that contains the `3', but that is impossible as cells are decreasing.

To show $\pi$ is simple, we again seek a contradiction and suppose that it contains a nontrivial block $B$.  The subsequence of $\pi$ consisting of all l-r maxes and all r-l mins is isomorphic to a simple permutation (see Figure \ref{SimpleExample}).  So if $B$ contains more than one point of this subsequence it must contain every point  and we would have $B=\pi$ which contradicts nontriviality.  Otherwise $B$ must be contained in a tile, but the tiles themselves are easily seen to be simple.  This final contradiction finishes the proof of the converse.
\end{proof}

\section{Enumeration of simple permutations in $\Av(2341, 4123, 3412)$}

As we saw in the previous section simple permutations  in $\Av(2341, 4123, 3412)$ of length more than 2  have their leftmost inflection point associated with either a l-r max or a r-l min and, as these two types are related by an inversion, there are equal numbers of each in every length.  So we shall enumerate those whose leftmost inflection is associated with a l-r max and then double the result.

We shall obtain the generating function of this set as a sum of terms,
%
with a typical term counting simple permutations in which there are $n$ extremal points (and therefore $n$ corner cells), a fixed set of $k$ interlacing corner pairs, and $t$ central cells (lying between non-interlacing corner cells) that can have 0 or 1 point.  The set of simple permutations of such a type is enumerated by the generating function
\begin{equation}\label{typicalterm}x^ny^kz^t\end{equation}
where $y = x^2/(1-x^2)$ and $z = 1 + x$.  This is because every interlacing pair of corner cells contributes some positive even number of points to the permutation while each central cell between non-interlacing corners contributes 0 or 1 point to the permutation.

Note that there are $n-3$ central cells because there is no central cell between the first and last pairs of corner cells.  So the value of $t$ depends on whether the first pair of corner cells and the last pair are among the set of $k$ interlacing pairs.  There are 3 different cases:

\begin{enumerate}
\item Both the first and last pairs of corner cells interlace.  Here $t=n-3-(k-2)=n-k-1$.
\item Only the first pair or the last pair of corner cells interlaces.  Here $t=n-3-(k-1)=n-k-2$.
\item Neither the first or the last pairs of corner cells interlace.  Here $t=n-3-k=n-k-3$.
\end{enumerate}

To find the number of choices for the $k$ interlacing pairs of corner cells in each of these 4 cases we make use of the following well known result.

\begin{lemma}
The number of ways of picking $\ell$ non-overlapping pairs $(i,i+1)$ from $\{1,\ldots,m\}$ is
\[\binom{m-\ell}{\ell}\]
\end{lemma}

%
In the first case of the above 3 possibilities two of the $k$ pairs are already chosen and the remaining $k-2$ pairs have to be chosen from the interior $n-4$ corner cells: this can be done in $\binom{n-4-(k-2)}{k-2}$ ways by the previous lemma.  
Similarly the second and third cases give, respectively, $2\binom{n-3-(k-1)}{k-1}$ and $\binom{n-2-k}{k}$ choices for selecting the $k$ interlacing pairs of corner cells.

Hence, for fixed $n$, the sum of all the terms in expression \eqref{typicalterm} over all choices of $k$ interlacing corner pairs is
\[\left(
{n-k-2\choose k-2} z^{n-k-1} + 2 {n-k-2 \choose k-1} z ^{n-k-2} + {n-k-2 \choose k} z^{n-k-3}\right)x^ny^k.
\]

So we need to sum the above expression 
over $n\ge 4$ (since there are no simples for $n\le3$) and $k\ge0$.  
This double summation is first summed over $n$ using the binomial expansion and then over $k$ as a geometric series (taking care of the boundary cases $k=0$ and $k=1$ separately).  The end result is
%

$$
\frac{x^4 z}{1-x z} + \frac{x^4 y}{(1 - x z)^2} + \frac{2 x^4 y z}{1 - x z} 
           +\frac{ x^4 y^2 z  (1 - x z + x)^2}{ (1 - x z)^2 (1- x z - x^2 y z) }.
$$
Expressing this in terms of $x$  and, multiplying by 2, we obtain
\begin{theorem}\label{Av234141233412enumeration}
The generating function for simple permutations of length greater than or equal to 4 in $\Av(2341, 4123, 3412)$ is
$$
 \frac{2(x^4 + x^6 + x^9)}{ (1 - x^2) (1 - 2 x + x^3 - x^4)}
$$
\end{theorem}


\section{The enumeration of $\Av(2341, 4123)$}

We now only have to assemble the pieces we have developed.  We begin by determining the allowed inflations in this class.

\begin{proposition}\label{prop-allowed-inflations}
Let $\sigma\in\Av(2341, 4123)$ be a simple permutation of length $m\ge 4$.  The inflation $\sigma[\alpha_1,\dots,\alpha_m]$ lies in $\Av(2341, 4123)$ if and only if every $\alpha_i$ is a decreasing sequence.
\end{proposition}
\begin{proof}
First, if each $\alpha_i$ is decreasing then any copy of $2341$ or $4123$ in $\sigma[\alpha_1,\dots,\alpha_m]$ could contain at most one entry from each $\alpha_i$, which is impossible because $\sigma$ itself avoids $2341$ and $4123$.

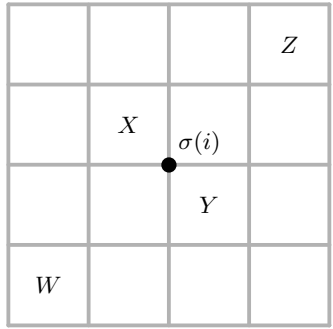
\begin{figure}
\begin{footnotesize}
\begin{center}
\psset{xunit=0.007in, yunit=0.007in}
\psset{linewidth=0.005in}
\begin{pspicture}(0,0)(240,240)
\multirput(0,0)(60,0){5}{\psline[linecolor=darkgray,linestyle=solid,linewidth=0.02in]{c-c}(0,0)(0,240)}
\multirput(0,0)(0,60){5}{\psline[linecolor=darkgray,linestyle=solid,linewidth=0.02in]{c-c}(0,0)(240,0)}
\pscircle*(120,120){0.04in}
\uput[45](120,120){$\sigma(i)$}
\rput[c](30,30){${W}$}
\rput[c](90,150){${X}$}
\rput[c](150,90){${Y}$}
\rput[c](210,210){${Z}$}
\end{pspicture}
\end{center}
\end{footnotesize}
\caption{The situation in the proof of Proposition~\ref{prop-allowed-inflations}: $\sigma(i)$ is not the `$2$' or `$3$' in a copy of $231$, nor the `$2$' in a copy of $312$, nor the `$1$' in a copy of $312$.}
\label{decreasinginflations}
\end{figure}

Now take $\sigma[\alpha_1,\dots,\alpha_m]\in\Av(2341,4123)$ and suppose to the contrary that $\alpha_i$ contains $12$ for some index $i$.  Then $\sigma(i)$ must not be the `$2$' or the `3' in a copy of $231$ (in $\sigma$), because that would lead to a copy of $2341$ in $\sigma[\alpha_1,\dots,\alpha_m]$.  Similarly, $\sigma(i)$ is neither the '$1$' nor the `$2$' in a copy of $312$.  Because $\sigma(i)$ is not the `$2$' in a copy of $231$, the entries of $\sigma$ to the right of $\sigma(i)$ must consist of a series of entries below $\sigma(i)$ followed by a entries of terms above $\sigma(i)$.  The other three conditions imply similar restrictions on the entries below $\sigma(i)$, the entries to the left of $\sigma(i)$, and the entries above $\sigma(i)$.  These restrictions are displayed in Figure~\ref{decreasinginflations} (following our conventions, unlabeled cells are empty in this diagram) which shows that $\sigma$ is 
sum decomposable if either $W$ or $Z$ is non-empty or skew decomposable if $W$ and $Z$ are both empty; in particular $\sigma$ is not simple.
\end{proof}

Because $\Av(2341, 4123)$ is sum closed, its generating function, which we label $f$, satisfies $f=1/(1-g)$, where $g$ denotes the generating function for nonempty sum indecomposable permutations in the class.  Therefore we need only determine $g$.  Because we will often be inflating permutations by decreasing sequences, it is convenient to define $d=x/(1-x)$.  Using Corollary \ref{simpletypes} the set of nonempty sum indecomposable permutations in $\Av(2341,4123)$ is then the union of
\begin{enumerate}
\item[(i).] the permutation $1$,
\item[(ii).] inflations of $21$ by permutations in $\Av(123)$, where the first entry is inflated by a skew indecomposable permutation,
\item[(iii).] inflations of  simple permutations of length at least $4$ in $\Av(123)=\Av(2341, 4123, 123)$ by decreasing sequences,
\item[(iv).] inflations of  simple permutations of length at least $4$ in $\Av(2341, 4123, 3412)$ by decreasing sequences, and
\item[(v).] inflations of $5274163$ by decreasing sequences.
\end{enumerate}
These sets are disjoint except for an intersection between those of types (iii) and (iv).  This intersection consists of the inflations of $123$-avoiding parallel alternations (see Figure~\ref{fig-example-simples}), of which there are two of every even length, so we have
$$
\mbox{g.f. for (iii) $\cap$ (iv)}=\frac{2d^4}{1-d^2}.
$$

Notice that the sets of type (i)--(iii) together comprise the set of non-empty sum indecomposable permutations of $\Av(123)$.  As every permutation is either sum indecomposable or sum decomposable, we can obtain the generating functions of these permutations by subtracting the generating function of sum decomposable permutations from the generating function for the  non-empty permutations of $\Av(123)$.  The sum decomposable permutations in $\Av(123)$ are inflations of $12$ by decreasing sequences, so we see
$$
\mbox{g.f. for (i)--(iii)}=c-d^2,
$$
where $c=(1-2x-\sqrt{1-4x})/(2x)$ is the generating function for the non-empty permutations in $\Av(123)$.

Finally, Theorem~\ref{Av234141233412enumeration} gives us the generating function for sum indecomposable permutations of type (iv):
$$
\mbox{g.f. for (iv)}=\frac{2(d^4+d^6+d^9)}{(1-d^2)(1-2d+d^3-d^4)},
$$
and sum indecomposable permutations of type (v) are counted by $d^7$.

Putting all these expressions together gives 

\begin{theorem}
The generating function $f$ for $\Av(2341, 4123)$ has the form $f=1/(1-g)$ where 
\begin{align*}
 g=&\frac{(1-2x-\sqrt{1-4x})}{2x}\\
&-\frac{(1-13x+74x^2-247x^3+539x^4-805x^5+834x^6-595x^7+283x^8-80x^9+8x^{10})x^2}{
(1-x)^7(1-2x)(1-6x+12x^2-9x^3+x^4)}
\end{align*}
\end{theorem}


Further calculations with a computer algebra package such as \textsf{Singular} shows that $f$  satisfies the quadratic

$$
\begin{array}{l}
\left(144x^{25}-3524x^{24}+38648x^{23}-259931x^{22}+1231750x^{21}-4420385x^{20}+12533805x^{19}\right.
\\
\quad-28844031x^{18}+54839380x^{17}-87179343x^{16}+116833299x^{15}-132706667x^{14}
\\
\quad+128169929x^{13}-105396633x^{12}+73761400x^{11}-43835832x^{10}+22029889x^9
\\
\quad-9301917x^8+3269458x^7-944215x^6+220007x^5-40293x^4+5578x^3-548x^2\\
\quad\left.+34x-1\right)f^2
\\
+\left(-48x^{25}+1380x^{24}-17556x^{23}+134339x^{22}-708318x^{21}+2775400x^{20}-8464162x^{19}\right.
\\
\quad+20701382x^{18}-41428652x^{17}+68785738x^{16}-95667058x^{15}+112183057x^{14}
\\
\quad-111372132x^{13}+93798415x^{12}-67025068x^{11}+40562377x^{10}-20710152x^9+8865879x^8
\\
\quad\left.-3153464x^7+920002x^6-216192x^5+39867x^4-5548x^3+547x^2-34x+1\right)f
\\
+\left(4x^{25}-132x^{24}+1921x^{23}-16624x^{22}+97464x^{21}-416740x^{20}+1361690x^{19}\right.
\\
\quad-3508914x^{18}+7290078x^{17}-12404442x^{16}+17480077x^{15}-20556472x^{14}+20271017x^{13}
\\
\quad-16800814x^{12}+11703343x^{11}-6835800x^{10}+3331377x^9-1343826x^8+443390x^7
\\
\quad\left.-117616x^6+24459x^5-3838x^4+427x^3-30x^2+x\right)
\\
=0.
\end{array}
$$

The \emph{growth rate} of the class $\C$ is the limit of $\sqrt[n]{|\C_n|}$ as $n\rightarrow\infty$ (if this limit exists).  In our case, this is the reciprocal of the least positive root of the discriminant of the minimal polynomial above, which is $4$, the same as the growth rate of $\Av(123)$.

\bibliographystyle{acm}
\bibliography{./refs}

\end{document}